\documentclass[letterpaper, 10 pt, conference]{ieeeconf}
\let\proof\relax
\let\endproof\relax
\IEEEoverridecommandlockouts
\usepackage{cite}
\usepackage{amsmath,amssymb,amsfonts}
\allowdisplaybreaks
\usepackage{algorithmic}
\usepackage{graphicx}
\usepackage{textcomp}
\usepackage{xcolor}
\usepackage{bbm}
\usepackage{balance}
\usepackage{algorithmic}
\usepackage{algorithm}
\usepackage{ragged2e}
\usepackage[caption=false, font=footnotesize]{subfig}
\newtheorem{theorem}{Theorem}[section]

\newtheorem{lemma}{Lemma}[section]
\newtheorem{lemmaappendix}{Lemma}[subsection]
\newtheorem{assumption}{Assumption}[section]
\newtheorem{remark}{Remark}[section]
\newtheorem{proof}{Proof}[section]

\def\BibTeX{{\rm B\kern-.05em{\sc i\kern-.025em b}\kern-.08em
    T\kern-.1667em\lower.7ex\hbox{E}\kern-.125emX}}
\begin{document}

\title{\LARGE \bf Decentralized Feedback Optimization via Sensitivity Decoupling: \\ Stability and Sub-optimality
}
%New command
\newcommand{\obj}{\Tilde{\Phi}}
\newcommand{\gest}{\Tilde{\nabla}}
\newcommand{\tr}[1]{\text{tr}(#1)}
\newcommand{\E}[1]{\bb{E}[#1]}
\newcommand{\Ef}[2]{\bb{E}[#1|\mathcal{F}_{#2}]}
\newcommand{\N}{\frac{1}{n}\mathbbm{1}^T\mathbbm{1}}
\newcommand{\inff}[1]{#1_{\infty}}
\newcommand{\bb}[1]{\mathbb{#1}}
\newcommand{\norm}[1]{\|#1\|}
\def\proof{\noindent\hspace{1em}{\itshape Proof: }}
\def\endproof{\hspace*{\fill}~\QED\par\endtrivlist\unskip}

\newcommand{\hzy}[1]{\textcolor{blue}{#1}}

\author{Wenbin Wang, Zhiyu He, Giuseppe Belgioioso, Saverio Bolognani, and Florian D{\"o}rfler
\thanks{The authors are with the Automatic Control Laboratory, ETH Z{\"u}rich, Physikstrasse 3, 8092, Z{\"u}rich, Switzerland. Emails: wenbiwang@student.ethz.ch, \{zhiyhe, gbelgioioso, bsaverio, dorfler\}@ethz.ch. This work was supported by the SNSF via NCCR Automation (grant number 180545). Z.~He also acknowledges the support of the Max Planck ETH Center for Learning Systems.}
}

\maketitle

\begin{abstract}
Online feedback optimization is a controller design paradigm for optimizing the steady-state behavior of a dynamical system. It employs an optimization algorithm as a dynamic feedback controller and utilizes real-time measurements to bypass knowing exact plant dynamics and disturbances. Different from existing centralized settings, we present a fully decentralized feedback optimization controller for networked systems to lift the communication burden and improve scalability. We approximate the overall input-output sensitivity matrix through its diagonal elements, which capture local model information. For the closed-loop behavior, we characterize the stability and bound the sub-optimality due to decentralization. We prove that the proposed decentralized controller yields solutions that correspond to the Nash equilibria of a non-cooperative game.
%We further prove that the resulting solution of the decentralized controller corresponds to the solution of a convex game.
Simulations for a voltage control problem on a direct current power grid corroborate the theoretical results.
\end{abstract}

% \begin{IEEEkeywords}
% Feedback optimization, networked system, decentralized optimization.
% \end{IEEEkeywords}

\section{Introduction}
Many engineering automation tasks consist of optimizing the steady-state operation of dynamical systems. Typical examples include congestion control in communication networks \cite{peterson2007computer}, voltage regulation \cite{Lukasortmann2020}, and optimal power flow in power systems \cite{Emiliano2016}. The key is to select an appropriate control input to optimize an objective function that reflects the input-output performance. This pursuit, however, involves some critical obstacles. Real-world systems are often large-scale, networked, and complex, making it challenging to precisely model their behaviors. Furthermore, exogenous disturbances often appear as parameters in steady-state optimization problems. Since disturbances are difficult to measure in general, numerical optimization relying on an exact formulation can be prohibitive or require conservative approximations.

Emerging feedback optimization controllers \cite{Hauswirth2021} exhibit huge potential in regulating dynamical plants in an efficient manner. The central idea is to use optimization algorithms as feedback controllers to drive the plant to an optimal steady-state operating point. It utilizes real-time measurements to update control inputs without requiring the exact plant model and disturbances. This feature endows feedback optimization with the versatility to handle various scenarios.
%Feedback optimization features versatility in handling various scenarios.
Examples include driving the system towards the global minimizer of convex functions \cite{simpsonporco2020}, stationary points of non-convex functions \cite{Hauswirthadrian2021, Miguelpicallo2020}, and competitive equilibria (e.g., Nash or Wardrop) of noncooperative games \cite{Belgioiosogiuseppe2021, Belgioioso2022}. Guarantees on closed-loop stability, robustness, and constraint satisfaction are established in continuous-time \cite{ Hauswirthadrian2021}, discrete-time \cite{Haberleverena2021}, and sampled-data \cite{Belgioiosogiuseppe2021} scenarios.

The closed-loop interconnection between dynamical plants and feedback optimization controllers brings unique challenges. When the plants are rapidly pre-stabilized, they can be abstracted by their steady-state input-output maps\cite{Miguelpicallo2020, Haberleverena2021}. Such abstractions help to obtain closed-loop guarantees. However, for general dynamical plants \cite{He2022}, real-time output measurements may not align with steady-state responses. This mismatch may cause non-negligible errors in stability analysis. To address this issue, one can employ singular perturbation analysis to quantify the required timescale separation for a satisfactory feedback optimization controller\cite{Hauswirthadrian2021}. The above methods require explicit knowledge of the steady-state input-output sensitivity matrix. To remove this restriction, recursive sensitivity learning based on streaming real-time measurements is proposed\cite{picallo2022adaptive}. Moreover, zeroth-order optimization algorithms utilizing gradient estimation \cite{Zhang2020} or Gaussian Processes \cite{ospina2022learning} offer an alternative that circumvents the need for sensitivity information. It has proven effective in feedback optimization \cite{He2022} and network optimization \cite{Tang2019, Tang2020}.

Centralized approaches may encounter various issues (e.g., in scalability and privacy) when deployed in large-scale networked systems. In this regard, distributed implementations of feedback optimization have been studied \cite{carnevale2023nonconvex}. These methods require that agents exchange local information with their neighbors in a network. To regulate the trade-off between global performance and local coordination, clustering strategies \cite{CHANFREUT202175} explore different aspects (e.g., partitioning \cite{8814851} and plug-and-play \cite{7040312}) to efficiently control large-scale networked systems. However, the establishment of communication channels may incur significant engineering effort, restrictions of feasible regions, time-consuming data transmission, and the need for iterative collaborative updates.

To overcome these limitations, we pursue a
fully decentralized and communication-free approach. Each agent uses the local sensitivity information and updates its control input without any communication. From a network-level perspective, this corresponds to approximating the sensitivity matrix by its diagonal elements, thereby leading to fully decoupled updates. We first show that the stationary point of this controller coincides with the Nash equilibrium of an underlying convex game. Then, we conduct a comprehensive analysis of the closed-loop stability and sub-optimality, namely, the distance between the globally optimal point and the stationary point to which the proposed controller converges. We characterize the dependence of the sub-optimality on the dynamic coupling and the properties of the objective.

% The rest of this paper is organized as follows. In Section \ref{section 2}, we provide preliminaries and the problem of interest. We present the decentralized controller and offer its game-theoretic interpretation in Section \ref{section 3}. The analysis of stability and sub-optimality is conducted in Section \ref{section 4}, where the interconnection with an algebraic map or a dynamic plant is analyzed. In Section \ref{section 5}, numerical simulations verify the theoretical results, followed by the conclusion in Section \ref{section 6}.
\section{Preliminaries and Problem Formulation}
\label{section 2}
\subsection{Notation}
Let $\bb{R}$ be the set of real numbers. We denote the inner product and the $l_2$-norm by $\left \langle \cdot,\cdot \right \rangle$ and $\norm{\cdot}$, respectively. For a differentiable function $\Phi:\bb{R}^n \times \bb{R}^n \rightarrow \bb{R}$, $\nabla_u \Phi(u,y)$ and $\nabla_y \Phi(u,y)$ denote the partial derivatives with respect to $u$ and $y$, respectively. For a square matrix $H$, we use $H_{\text{diag}}$ to represent the diagonal matrix with the diagonal elements of $H$. The largest singular value of $H$ is $\sigma_{\max}(H)$, and the smallest one is $\sigma_{\min}(H)$. We denote the realization of a vector $x \in \bb{R}^n$ at time $k$ by $x_k$, and its $i$-th entry by $x_{i,k}$.

\subsection{Problem Formulation}
In this paper, we focus on feedback optimization for a networked system with $N$ agents. The plant is represented by an asymptotically stable linear time-invariant (LTI) system
\begin{equation}
\label{LTI}
\begin{split}
    x_{k+1} &= Ax_k + Bu_k,\\
        y_k     &= Cx_k + Du_k + d,
\end{split}  
\end{equation}
where $x \in \bb{R}^N$ is the state, $u \in \bb{R}^N$ is the input, $y \in \bb{R}^N$ is the output, and $d\in \bb{R}^N$ is a constant unknown disturbance. Every element of $x$, $y$, and $u$ is the local state, output, and input of the corresponding agent. The plant \eqref{LTI} has a linear steady-state input-output map $y=Hu+d$, where $H=C(I-A)^{-1}B+D\in \bb{R}^{N\times N}$ is the sensitivity matrix. The scalar input and output of each agent are considered for notational convenience. In fact, the design and analysis can be easily extended to the multivariate setting with decoupled input constraint sets. An example of a system with three agents is given in Fig.~\ref{fig: problem formulation}. 

\begin{figure}[!t]
    \centering
    \includegraphics[width=0.5\textwidth]{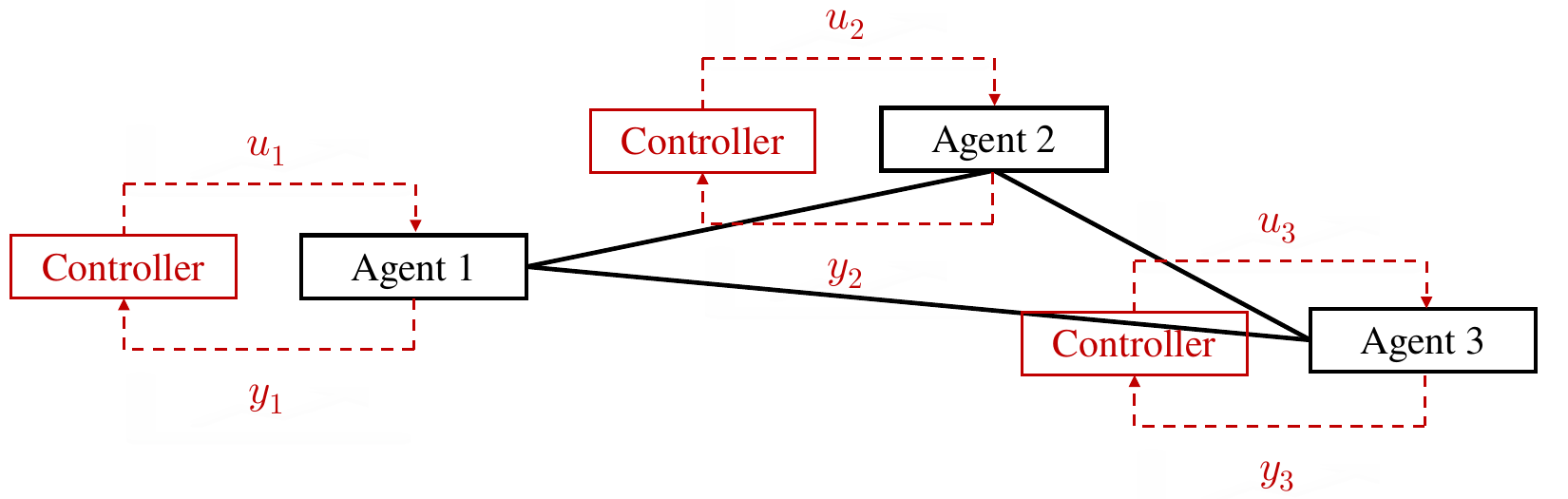}
    \caption{Example of a system with three agents. The black line indicates the coupling dynamics between agents. The red dashed line represents the decentralized update using local inputs and outputs.}
    \label{fig: problem formulation}
\end{figure}

We aim to drive the plant \eqref{LTI} to an optimal steady-state operating point defined by
\begin{subequations}
    \label{prob: introduction 2}
    \begin{align}
    \min_{u,y} \quad & \Phi(u,y)\\
    \textrm{s.t.} \quad & y=Hu + d, \label{eq:ss-map}
    \end{align}
\end{subequations}
where the global objective function $\Phi(u,y)=\sum_{i=1}^{N}\Phi_i(u_i,y_i)$ is the sum of all the local objectives $\Phi_i(u_i,y_i)$, and $u_i$ and $y_i$ are the input and output of agent $i$, respectively. By replacing $y$ in the objective, problem \eqref{prob: introduction 2} is transformed to the following unconstrained problem:
\begin{equation}
    \label{prob: introduction 3}
    \min_{u} \quad \obj(u),
\end{equation}
where $\obj(u)=\sum_{i=1}^{N}\obj_i(u)$, and $\tilde \Phi_i(u) = \Phi_i(u_i,y_i)$. We make the following assumptions about the objective function.
\begin{assumption}
\label{assumption: chapter 2}
The following conditions hold:
    \begin{enumerate}
        \item Agent $i$'s objective function $\Phi_i(u_i,y_i):\bb{R} \times  \bb{R}\rightarrow \bb{R}$ is separable with respect to $u_i$ and $y_i$, i.e., $\Phi_i(u_i,y_i)=\Phi_i^{(1)}(u_i) + \Phi_i^{(2)}(y_i)$;
        \item $\Phi_i^{(1)}(u_i):\bb{R} \rightarrow \bb{R}$ is continuously differentiable, $L_u$-smooth, and $m_u$-strongly convex;
        \item $\Phi_i^{(2)}(y_i):\bb{R} \rightarrow \bb{R}$ is continuously differentiable, $L_y$-smooth, and $m_y$-strongly convex.
    \end{enumerate}
\end{assumption}
The assumptions on smoothness and strong convexity are common in the literature of numerical optimization and control \cite{simpsonporco2020, hauswirth2020antiwindup}. Furthermore, the assumption on separable objective functions is largely satisfied in practical applications \cite{Emiliano2016}. It follows by Assumption \ref{assumption: chapter 2} that the global objective functions $\Phi^{(1)}(u)=\sum_{i = 1}^{N}\Phi_i^{(1)}(u_i)$ and $\Phi^{(2)}(y)=\sum_{i = 1}^{N}\Phi_i^{(2)}(y_i)$ are $NL_u$ and $NL_y$-smooth, respectively. The objective function $\obj(u)$ is $L$-smooth and $m$-strongly convex, where $L=NL_u+N\sigma_{\max}^2(H)L_y$ and $m=Nm_u+N\sigma_{\min}^2(H)m_y$, respectively.

\section{Design of the Decentralized Controller}
\label{section 3}
Existing centralized feedback optimization controllers exploit the following gradient-based iterative update\cite{Hauswirth2021}
\begin{equation}
\label{centralized update for agent}
    u_{k+1}=u_k-\eta \left(\nabla_u \Phi(u_k,y_k)+H^{\top}\nabla_y \Phi(u_k,y_k)\right),
\end{equation}
where $y_k$ is the real-time measurement, and $\eta>0$ is the step size. Note that the update in  \eqref{centralized update for agent} involves the full sensitivity matrix $H$, where its element $H_{ij}$ represents the sensitivity of agent $j$'s output to agent $i$'s input. In practice, such a coupling sensitivity requires subsystems to exchange a large amount of information at each sampling instant $k$. In scenarios where i) communication is hard to establish (due to cost, privacy, etc.), ii) the sampling period is too short for communication to occur, or iii) the cross-coupling terms defining how agents influence one another are unknown, the centralized controller \eqref{centralized update for agent} may not be implementable in practice. Instead, we propose a fully decentralized controller via an approximate and decoupled sensitivity matrix.

\subsection{Decentralized Controller}
To develop a decentralized counterpart of \eqref{centralized update for agent}, each agent updates its control input only based on its local sensitivity and measurement, i.e.,
\begin{equation*}
\begin{split}
    u_{i,k+1}=u_{i,k}-\eta\left(\nabla_u \Phi_i(u_{i,k},y_{i,k})+H_{ii}\nabla_y \Phi_i(u_{i,k},y_{i,k})\right),
\end{split}
\end{equation*}
where the local sensitivity information $H_{ii}$ and the local measurement $y_{i,k}$ are utilized to calculate $u_{i,k+1}$.
From a system perspective, the update is represented by 
\begin{equation}
\label{eq: decentralized algorithm system}
    u_{k+1}=u_k-\eta \left(\nabla_u \Phi(u_k,y_k)+ H_{\text{diag}}^{\top}\nabla_y \Phi(u_k,y_k)\right).
\end{equation}
By using the diagonal matrix $H_{\text{diag}}$ to approximate the sensitivity matrix $H$, this controller is fully decoupled over the network. Hence, each agent can perform local adjustments without any communication.

\subsection{A Game-Theoretic Interpretation}
Next, we analyze the stationary points $\inff{u}$ to which the decentralized controller \eqref{eq: decentralized algorithm system} converges. 
% by which we denote $\inff{u}$.
\begin{theorem}
\label{theorem: game theory}
    Given Assumption \ref{assumption: chapter 2}, the stationary points of the controller \eqref{eq: decentralized algorithm system}, if they exist, equal to the Nash equilibria of the following convex game:
    \begin{equation}
    \label{convex game}
    \begin{aligned}
    \forall i, \quad\min_{u_i} \quad & \obj_i(u).\\
    \end{aligned}
\end{equation}
\end{theorem}

\begin{proof}
    The $i$-th entry of the fixed point $\inff{u}$ of \eqref{eq: decentralized algorithm system} satisfies 
    \begin{equation*}
        \label{game map}
        \nabla_u \Phi_i(u_i,(H\inff{u})_i+d_i)+H_{ii}\nabla_y \Phi_i(u_i,(H\inff{u})_i+d_i)=0.
    \end{equation*}
    The system-level representation of the above update is
    \begin{equation*}
    \label{stationary condition: decentralized algo}
        \nabla_u \Phi(\inff{u},H\inff{u}+d)+ H_{\text{diag}}^{\top}\nabla_y \Phi(\inff{u},H\inff{u}+d)=0.
    \end{equation*}
    The left-hand side of the above equation is also the pseudo-gradient mapping, i.e., $u \mapsto [\partial_{u_1} \obj_1(u), \cdots, \partial_{u_N}\obj_N(u)]^{\top}$, of the game \eqref{convex game}, whose zeros are the Nash equilibria \cite[Theorem 5]{facchinei2010generalized}. Thus, the stationary points of \eqref{eq: decentralized algorithm system} are the Nash equilibria of the game \eqref{convex game}. 
\end{proof}
\section{Performance Analysis}
\label{section 4}
%Note that 
The update direction on the right-hand side of \eqref{eq: decentralized algorithm system} does not align with the gradient of the objective function \eqref{prob: introduction 3}. To study the stability and sub-optimality of the interconnection of the controller \eqref{eq: decentralized algorithm system} and the plant, we establish the condition of the plant such that the iteration of this decentralized controller is strongly monotone \cite[Def. 22.1]{Bauschke}. We develop the analysis for two separate cases, i.e., the plant is represented by the steady-state map \eqref{eq:ss-map} or by the LTI plant \eqref{LTI}. 

\subsection{Interconnection with an Algebraic Plant}
We first consider the case where the plant has fast-decaying dynamics given any fixed inputs. Then, the plant can be replaced by its linear steady-state map \eqref{eq:ss-map} \cite{Haberleverena2021}. The closed-loop system with the decentralized controller \eqref{eq: decentralized algorithm system} is
\begin{subequations}
\label{sys: wo dynamic}
\begin{align}
    \label{sys: wo dynamic a}
    y_k &= Hu_k + d,\\
    \label{sys: wo dynamic b}
    u_{k+1} &= u_k - \eta (\nabla \Phi^{(1)}(u_k) + H_{\text{diag}}^{\top}\nabla \Phi^{(2)}(y_k)).
\end{align}
\end{subequations}
Furthermore, the right-hand side of \eqref{sys: wo dynamic b} is strongly monotone when the plant \eqref{sys: wo dynamic a} is weakly coupled, see the characterization \eqref{weakly coupled} below.
\begin{lemma}
\label{lemma: unique stationary point}
    Given Assumption \ref{assumption: chapter 2}, we conclude that $\nabla \Phi^{(1)}(u) + H_{\text{diag}}^{\top}\nabla \Phi^{(2)}(Hu+d)$ is a $(m-c)$-strongly monotone operator if $m>c$, where $m=Nm_u+N\sigma_{\min}^2(H)m_y$ and $c=N\sigma_{\max}(H-H_{\text{diag}})\sigma_{\max}(H)L_y$.
\end{lemma}
\begin{proof}
The proof can be found in Appendix \ref{Proof for lemma}.
\end{proof}

Note that the condition $m>c$ is equivalent to
\begin{equation}
\label{weakly coupled}
    \sigma_{\max}(H-H_{\text{diag}})\leq\frac{m_u+\sigma_{\min}^2(H)m_y}{\sigma_{\max}(H)L_y}.
\end{equation}
A diagonally dominant sensitivity matrix $H$ can help to satisfy this condition.

If the plant satisfies \eqref{weakly coupled}, the stationary point $\inff{u}$ of the closed-loop interconnection \eqref{sys: wo dynamic} 
exists and is unique \cite[Corollary 23.37]{Bauschke}. Let $u^*$ be the globally optimal point of \eqref{prob: introduction 2} and $y^* = Hu^* + d$ be the corresponding steady-state output. The following theorem shows that the system \eqref{sys: wo dynamic} converges to a neighborhood of $u^*$ with a linear rate. 

\begin{theorem}
\label{theorem: no plant dynamic strongly convex}
    Let Assumption 2.1 hold, and let the plant \eqref{sys: wo dynamic a} satisfy the diagonal dominance \eqref{weakly coupled}. For all $\eta \in (0,\frac{2m-2c}{L^2-m^2})$, the closed-loop system \eqref{sys: wo dynamic} satisfies
    \begin{equation}
    \label{eq: stability result 1}
    \|u_{k}\!-\!u^*\| \leq \rho^k\|u_0\!-\!u^*\|\!+\!\eta\|(H^{\top}\!-\!H_{\text{diag}})\nabla \Phi^{(2)}(y^*)\|\sum_{j=0}^{k-1}\rho^j,
    \end{equation}
where $\rho=\sqrt{1-2m\eta +L^2\eta^2}+c\eta$, $0<\rho<1$, and $m$ and $c$ are given in Lemma~\ref{lemma: unique stationary point}.
%$c=N\sigma_{\max}(H-H_{\text{diag}})\sigma_{\max}(H)L_y$. 
\end{theorem}

\begin{proof}
Please see Appendix \ref{proof for theorem 1}.
\end{proof}

To characterize the sub-optimality, we bound the distance between the globally optimal point $u^*$ and the limit point $\inff{u}$ to which the controller \eqref{eq: decentralized algorithm system} converges. The following theorem provides this upper bound.

\begin{theorem}
\label{theorem: no plant dynamic strongly convex vanilla}
    If $2m>1$ and the conditions in Theorem~\ref{theorem: no plant dynamic strongly convex} hold,
    %the plant \eqref{sys: wo dynamic a} satisfies \eqref{weakly coupled}, 
    then we have 
    \begin{equation}
    \label{bound: 2}
        \norm{u^*-\inff{u}} \leq \norm{(H^{\top}-H_{\text{diag}})\nabla \Phi^{(2)}(\inff{y})} \sqrt{\frac{1}{2m-1}}.
    \end{equation}
\end{theorem}
\begin{proof}
    Please see Appendix \ref{appendix: proof 4}.
\end{proof}

We obtain the bound \eqref{bound: 2} by analyzing the trajectory of the centralized controller \eqref{centralized update for agent} while taking the limit point $\inff{u}$ of the decentralized controller \eqref{eq: decentralized algorithm system} as a reference. Note that $2m>1$ is a sufficient condition for establishing closed-loop stability from this perspective.
%for the stability of the interconnection of the centralized controller \eqref{centralized update for agent} and the algebraic plant \eqref{sys: wo dynamic a}. 

\begin{remark}
    The sub-optimality of the controller \eqref{eq: decentralized algorithm system} depends on the coupling degree $H^{\top}-H_{\text{diag}}$ and the properties of objective functions (e.g., $\nabla\Phi^{(2)}(y)$ and $m$). The bound decreases as $m$ increases. It provides an estimate of the distance between the Nash equilibrium and the globally optimal solution. We will illustrate the tightness of this bound via simulations later in Section \ref{section 5}.
\end{remark}

\subsection{Interconnection with an LTI Plant}
Consider the closed-loop interconnection of the LTI plant \eqref{LTI} and the controller \eqref{eq: decentralized algorithm system}
\begin{subequations}
\label{sys: LTI}
    \begin{align}
    \label{LTI_x}
        x_{k+1} &= Ax_k + Bu_k,\\
        \label{LTI_y}
        y_k     &= Cx_k + Du_k + d,\\
        u_{k+1} &= u_k-\eta( \nabla \Phi^{(1)}(u_k) + H_{\text{diag}}\nabla\Phi^{(2)}(y_k)).
    \end{align}
\end{subequations}
The existence of plant dynamics will cause a difference between the real-time measurement and the steady-state response. Hence, we consider the coupled errors involving the distance to the steady state of the plant \eqref{LTI} given a fixed input (i.e., $\norm{x-H_xu}^2$) and the distance to the equilibrium point of the controller \eqref{eq: decentralized algorithm system} (i.e., $\norm{u-\inff{u}}^2$). These two coupled dynamics decay to zero with a linear rate if the steady-state map of the LTI plant \eqref{LTI} satisfies the diagonal dominance \eqref{weakly coupled}.

\begin{theorem}
\label{theorem: plant dynamic}
    Suppose that the stable LTI plant \eqref{LTI} satisfies \eqref{weakly coupled} and that Assumption \ref{assumption: chapter 2} holds. Then, there exists a positive constant $\eta^*$ such that $\forall \eta \in (0,\eta^*)$, the closed-loop system \eqref{sys: LTI} exhibits a linear convergence, i.e.,
    \begin{equation}
    \label{LTI result}
    \left\Vert
    \begin{matrix}
       x_{k}-H_xu_{k} \\
       u_{k}-\inff{u}\\
    \end{matrix}
    \right\Vert^2
    \leq (\lambda_{\max}(\Xi))^k
    \left\Vert
    \begin{matrix}
       x_0-H_xu_0 \\
       u_0-\inff{u}\\
    \end{matrix}
    \right\Vert^2,
    \end{equation}
    where $H_x=(I-A)^{-1}B$, $0<\lambda_{\max}(\Xi)<1$, $\Xi$ is a symmetric matrix given in Appendix \ref{appendix: proof 5}, and $\eta^*$ is given in Appendix \ref{appendix: proof 5}. %and $m$ and $c$ are given in Lemma~\ref{lemma: unique stationary point}.
\end{theorem}

\begin{proof}
    Please see Appendix \ref{appendix: proof 5}.
\end{proof}

Note that the steady-state inputs of systems \eqref{sys: LTI} and \eqref{sys: wo dynamic} are identical. By Lemma \ref{lemma: unique stationary point}, if \eqref{weakly coupled} holds, the stationary point $\inff{u}$ of the system \eqref{sys: wo dynamic} is unique. It means that the system \eqref{sys: LTI}, if it converges, will converge to the stationary point of the system \eqref{sys: wo dynamic}. Hence, the upper bound on the distance between $\inff{u}$ and the globally optimal point $u^*$ in \eqref{bound: 2} still holds.

\section{Numerical Results}
\label{section 5}
We consider a voltage control problem for a direct current (DC) power system. The model is adapted from \cite{JinxinZhao}, and it is illustrated in Fig. \ref{fig: chapter 2 8 node system}. This system consists of 8 nodes and 9 edges. There are two hub-like structures with a connection between nodes 4 and 5. The system dynamics are
\begin{equation}
\begin{split}
    \label{eq: power system formulation}
    \begin{bmatrix}
        C&0\\
        0&L
    \end{bmatrix}
    \begin{bmatrix}
        \dot{V}\\
        \dot{f}
    \end{bmatrix}=&
    \begin{bmatrix}
        G&-B\\
        B^\top&-R
    \end{bmatrix}
    \begin{bmatrix}
        V\\
        f
    \end{bmatrix}+
    \begin{bmatrix}
        I^* + I_c\\
        0
    \end{bmatrix},\\
    V_{\text{m}}=&V+d,
\end{split}
\end{equation}
where $V \in \bb{R}^8$ is the node voltage, $V_{\text{m}} \in \bb{R}^8$ is the measured voltage with the unknown measurement error $d\in \bb{R}^8$, $f\in \bb{R}^8$ is the line current, $I^*\in \bb{R}^8$ is the reference current injection, and $I_c\in \bb{R}^8$ is the vector consisting of controllable current injection at each node. There are two diagonal matrices $C\in \bb{R}^{8\times8}$ and $G\in \bb{R}^{8\times8}$, where the diagonal elements represent the capacitance and the resistance of each node, respectively. Moreover, $L\in \bb{R}^{9\times9}$ and $R\in \bb{R}^{9\times9}$ are also diagonal matrices, where the diagonal elements indicate the inductance and the resistance of each line, respectively. The incidence matrix of the electrical network is $B\in \bb{R}^{8\times9}$. 

Given an unknown change $\Delta I \in \mathbb{R}^8$ in current injection $I^*$, the objective is to operate the system such that the voltage measurement $V_{\text{m}}$ at each node tracks the reference value $V_{\text{m, ref}}$ with a minimal control effort related to $I_c$. Furthermore, we consider scenarios where communication is unfavorable (e.g., to avoid attacks). Hence, the current injection at each node is independently decided based on locally available measurements. This objective is 
% formalized as
\begin{equation}
    \label{prob: quadratic function DC power system}
    \begin{aligned}
    \min_{I_c,V_{\text{m}}} \quad & \frac{1}{2}(\gamma_1\norm{I_c}^2+\gamma_2\norm{V_{\text{m}}-V_{\text{m, ref}}}^2)\\
    \textrm{s.t.} \quad & V_{\text{m}}=H(I_c+I^*-\Delta I)+d,
    \end{aligned}
\end{equation}
where the input-output steady-state sensitivity matrix $H =
    \begin{bmatrix}
        I&0
    \end{bmatrix}
    \begin{bmatrix}
        G&-B\\
        B^\top&-R
    \end{bmatrix}^{-1}
    \begin{bmatrix}
        I\\
        0
    \end{bmatrix}\in \bb{R}^{8\times 8}$, and $V_{\text{m, ref}}=HI^*+d$.
We set $C$ and $L$ to be the identity matrix. We set $I^*$ and $\Delta I$ to be all-ones vectors. The resistor $R$ at each line equals $10$. We choose different values of $G$ to analyze cases with different levels of coupling (see \eqref{weakly coupled}).
The resulting system \eqref{eq: power system formulation} is stable \cite{JinxinZhao}. We further discretize the system \eqref{eq: power system formulation} using Euler forward discretization with a step-size $\epsilon=0.1$. We set $\gamma_1=\gamma_2=1$ and select the step size used by the controllers $\eta = 0.05$, which meets the stability condition provided in Theorem \ref{theorem: no plant dynamic strongly convex}.
%the discretization parameter $\epsilon=0.1$

\begin{figure}[!t]
    \centering
    \includegraphics[width=0.6\linewidth]{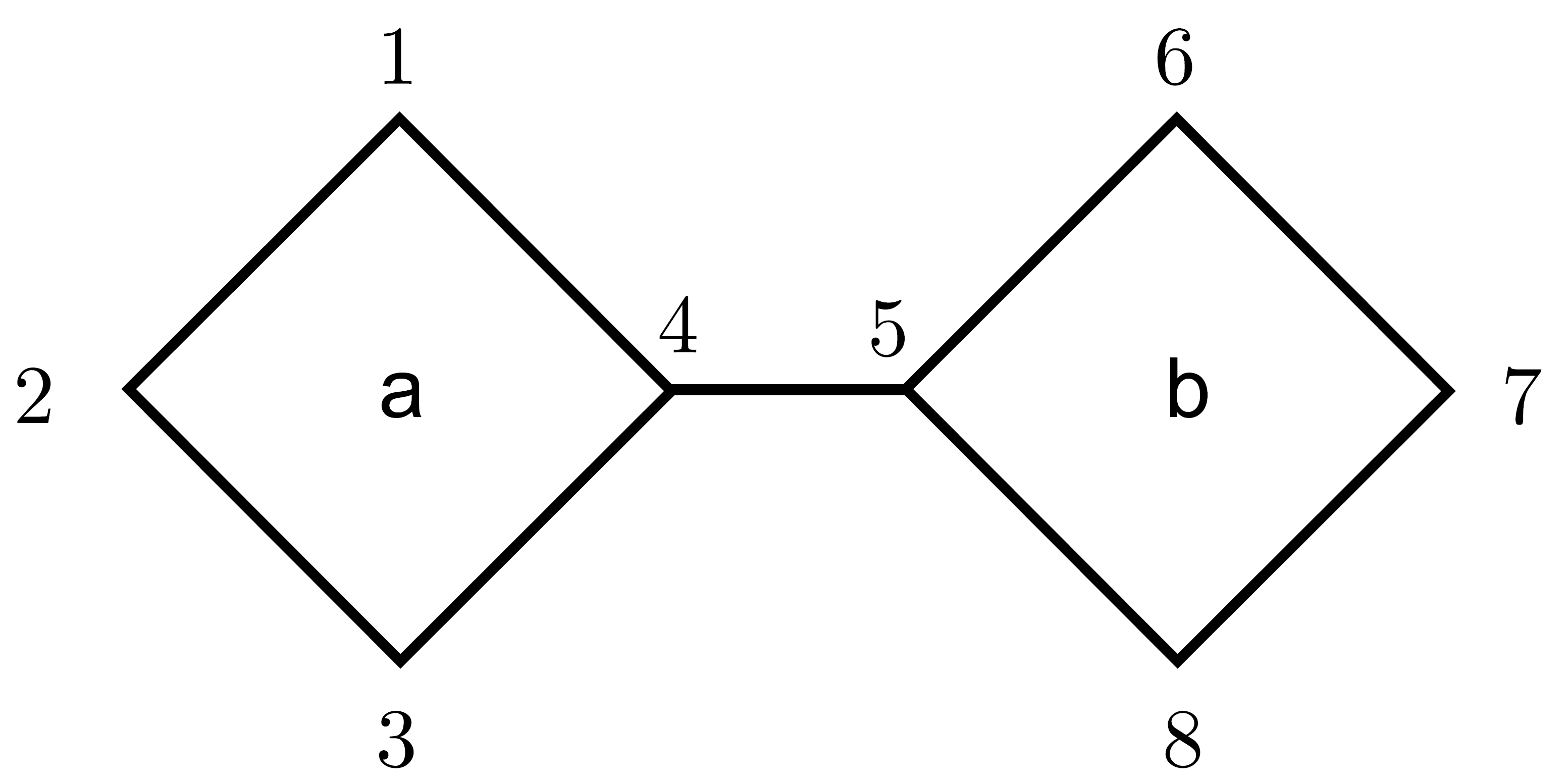}
    \caption{An 8-node DC power system.}
    \label{fig: chapter 2 8 node system}
\end{figure}

In Fig. \ref{fig: chapter 2 8 node system input error (G=1)}, the relative errors (i.e., $\norm{u_k-u^*}/\norm{u^*}$) for the input $u \triangleq I_c$ are plotted when $G=1$. We observe that the normal gradient method (i.e., the centralized controller) brings the system to the optimal operating point, whereas the decentralized controller \eqref{eq: decentralized algorithm system} leads to a sub-optimal solution. Furthermore, in Fig. \ref{fig: G}, we illustrate the evolution of the upper bound \eqref{bound: 2} and the corresponding true sub-optimality when the parameter $G$ varies from 1 to 100. As the degree of diagonal dominance of the sensitivity matrix $H$ increases (i.e., $G$ increases from 1 to 100), we observe a decrease in sub-optimality, which aligns with the results outlined in Theorem \ref{theorem: no plant dynamic strongly convex vanilla}. Notably, the bound \eqref{bound: 2} provides a rather accurate estimation of the sub-optimality associated with the decentralized controller. From the perspective of the system operator, the bound \eqref{bound: 2} can act as a surrogate reference for the price of decentralization, namely,  the extra cost incurred by disregarding the dynamic cross-coupling.

\begin{figure}[t!]
    \subfloat[Algebraic steady-state map.\label{fig: chapter 2 8 node system input error al}]{
       \includegraphics[width=0.48\linewidth]{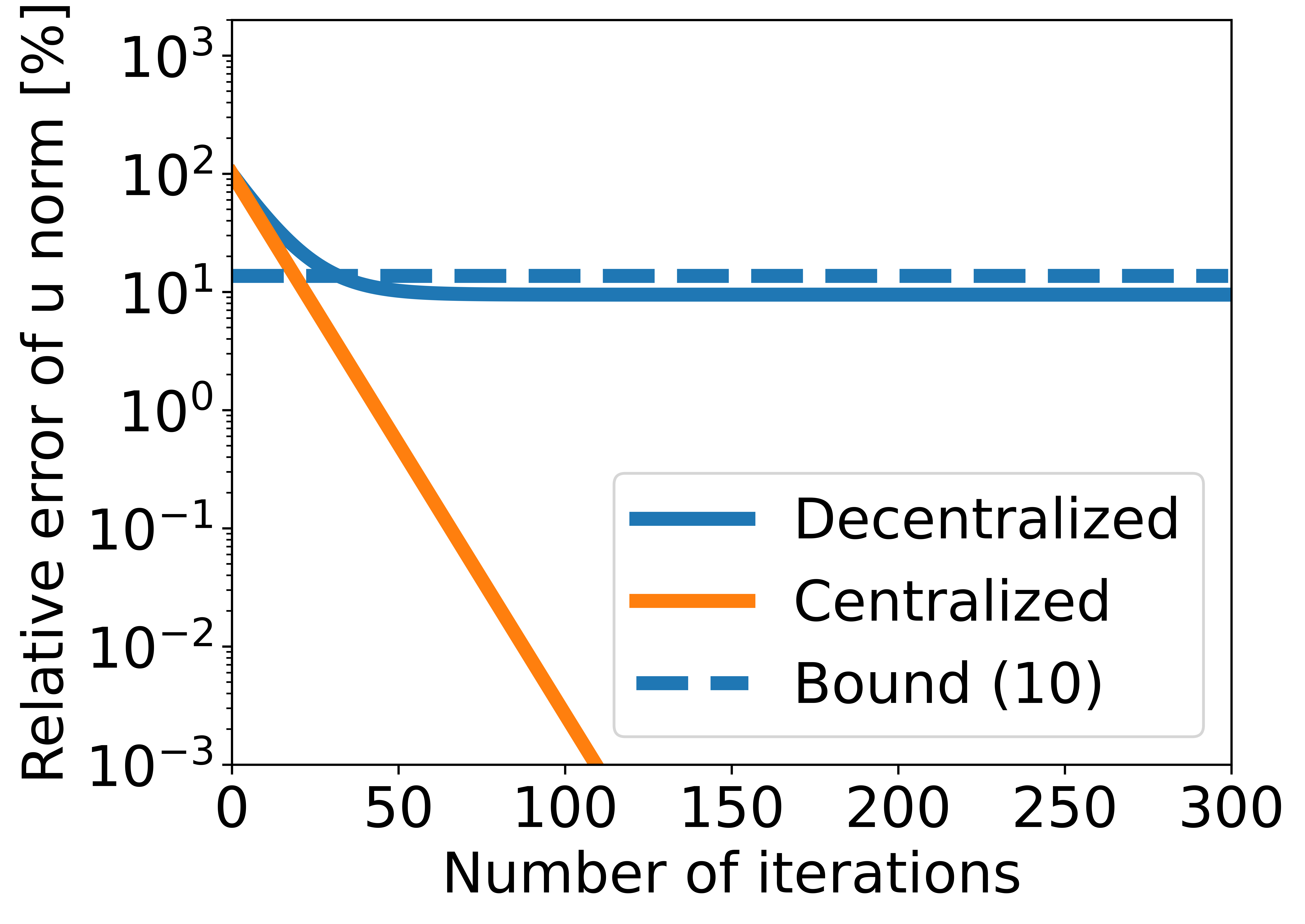}}
    \hfill
    \subfloat[LTI plant.\label{fig: chapter 2 8 node system input error da}]{
        \includegraphics[width=0.48\linewidth]{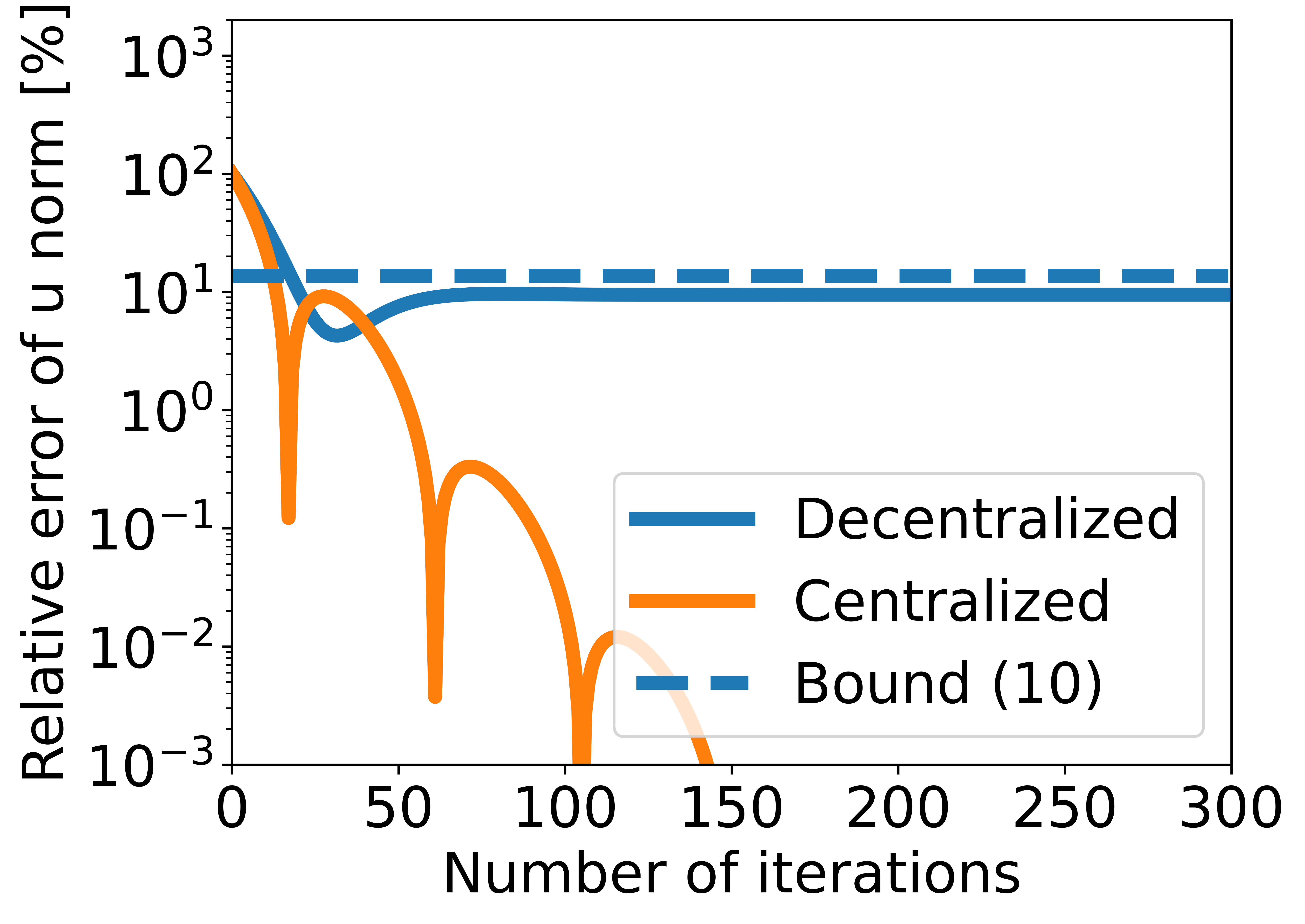}}
    \caption{Relative distance to the globally optimal point (G=1).}
    \label{fig: chapter 2 8 node system input error (G=1)}
\end{figure}

\begin{figure}[t!]
\centering
\includegraphics[width=0.7\linewidth]{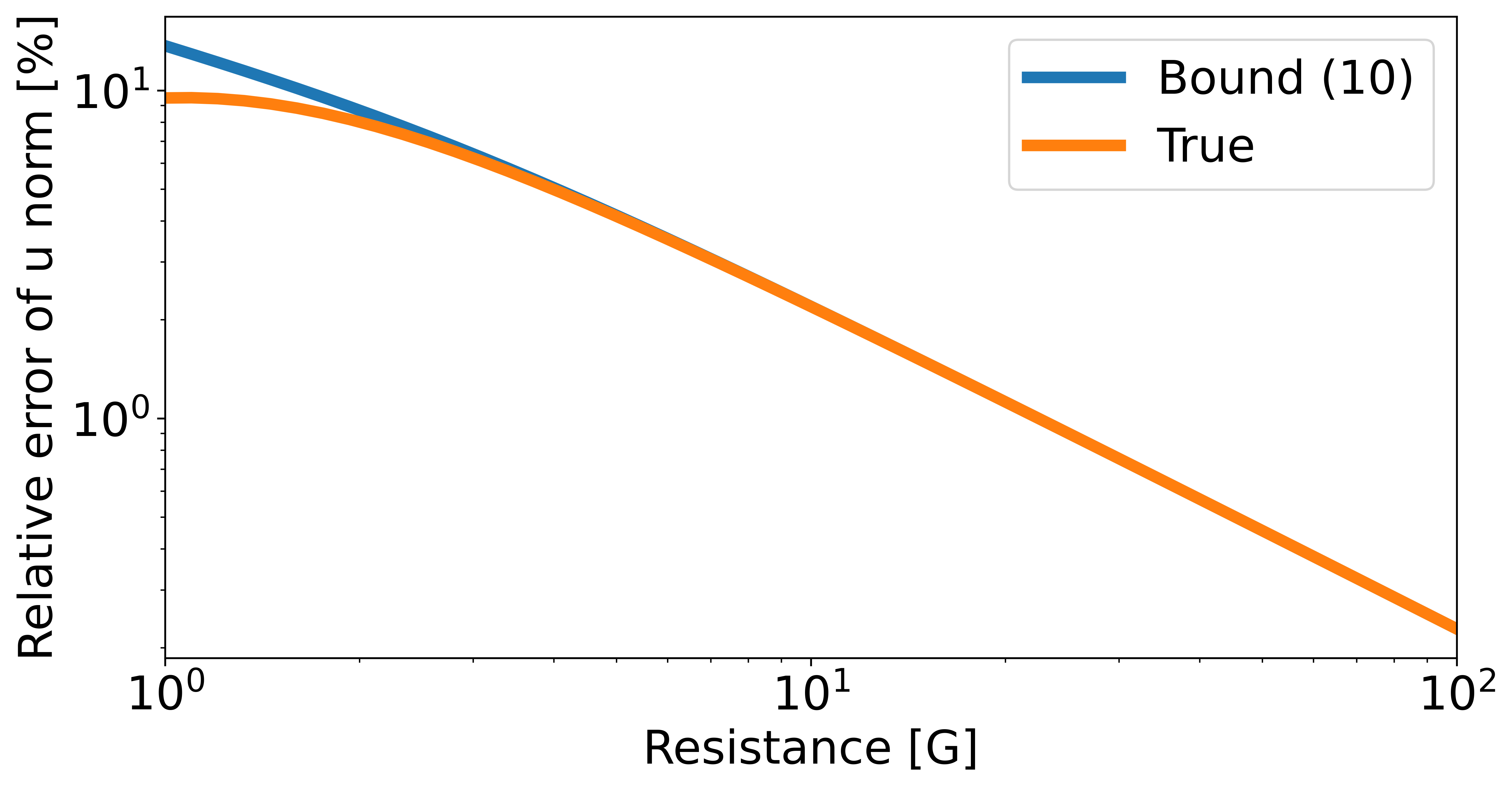}
\caption{Relative distance to the globally optimal point with different values of $G$ (i.e., different degrees of diagonal dominance).}
\label{fig: G}
\end{figure}
\section{Conclusion}
\label{section 6}
We considered feedback optimization for networked systems, where inter-agent communication may be expensive or hard to establish. We proposed a fully decentralized controller through an approximated sensitivity matrix based on local model information. We proved that the resulting steady-state operating points are Nash equilibria. Moreover, we derived sufficient conditions for the stability of the closed-loop systems. We explicitly characterized the sub-optimality of the decentralized controller, reflecting the price of decentralization. Through simulations on a DC power grid, we illustrated the performance of the decentralized controller and the tightness of the bound on sub-optimality.

% In future work, one direction is to explore the extension to handle non-smooth and non-convex functions. Another direction is to study the robustness of the proposed controller against inaccurate local model information. Finally, it would also be meaningful to explore diverse application scenarios, e.g., robotics and traffic networks.

\bibliographystyle{IEEEtran}
\bibliography{references}

\balance
\appendix
\subsection{Proof of Lemma \ref{lemma: unique stationary point}}
\label{Proof for lemma}
For $\forall u_1, u_2 \in \bb{R}^N$, we have
    \begin{align*}
        &(\nabla\Phi^{(1)}(u_1) + H_{\text{diag}}^{\top}\nabla \Phi^{(2)}(Hu_1+d)\\
        &-\nabla\Phi^{(1)}(u_2) - H_{\text{diag}}^{\top}\nabla \Phi^{(2)}(Hu_2+d))^{\top}(u_1-u_2)\\
        \overset{\text{(s.1)}}{\geq}& Nm_u\norm{u_1\!-\!u_2}^2 \!+\! \left(\nabla \Phi^{(2)}(Hu_1\!+\!d)\!-\!\nabla \Phi^{(2)}(Hu_2\!+\!d)\right)^{\top}\\
        &\cdot(H_{\text{diag}}(u_1\!-\!u_2)+Hu_1\!+\!d-(Hu_2\!+\!d)-H(u_1\!-\!u_2))\\
        \overset{\text{(s.2)}}{\geq}& (Nm_u+Nm_y\sigma_{\min}^2(H))\norm{u_1-u_2}^2\\
        &-\sigma_{\max}(H_{\text{diag}}-H)\sigma_{\max}(H)NL_y\norm{u_1-u_2}^2\\
        =&\left(Nm_u\!+\!Nm_y\sigma_{\min}^2(H)\!-\!\sigma_{\max}(H_{\text{diag}}\!-\!H)\sigma_{\max}(H)NL_y\right)\\
        &\cdot\norm{u_1-u_2}^2\\
        =&(m-c)\norm{u_1-u_2}^2,
    \end{align*}
where in (s.1), we use the strong convexity of $\Phi^{(1)}(u)$ and reorganize the terms. The strong convexity and the smoothness of $\Phi^{(2)}(y)$ are utilized in (s.2).

\subsection{Proof of Theorem \ref{theorem: no plant dynamic strongly convex}}
\label{proof for theorem 1}
The sub-optimality of the input $u$ at time step $k+1$ satisfies
\begin{equation*}
\label{no dynamic strongly convex proof: 1}
   \begin{split}
    \|&u_{k\!+\!1}\!-\!u^*\|=\|u_k\!-\!u^*\!-\!\eta \big(\nabla \obj(u_k)\!-\!(H^{\top}\!-\!H_{\text{diag}})\\
    &\quad\cdot(\nabla \Phi^{(2)}(y_k)-\nabla \Phi^{(2)}(y^*)+\nabla \Phi^{(2)}(y^*))\big)\|\\
    &\overset{\text{(s.1)}}{\leq} \|u_k-u^*-\eta \nabla \obj(u_k)\|\!+\!\eta\|(H^{\top}\!-\!H_{\text{diag}})\nabla \Phi^{(2)}(y^*)\|\\
    &\quad+\eta\sigma_{\max}(H^{\top}-H_{\text{diag}})NL_y\sigma_{\max}(H)\|u_k-u^*\|\\
    &\overset{\text{(s.2)}}{\leq}(\sqrt{1-2m\eta+L^2\eta^2}+c\eta)\|u_k-u^*\|\\
    &\quad+\eta\|(H^{\top}-H_{\text{diag}})\nabla \Phi^{(2)}(y^*)\|,\\
\end{split} 
\end{equation*}
where the triangle inequality and the smoothness of $\Phi^{(2)}(y)$ are utilized in (s.1). In (s.2),  the strong convexity and smoothness of $\obj(u)$ are used to bound $\|u_k-u^*-\eta \nabla \obj(u_k)\|$. After reorganizing the terms, we obtain the above bound. Telescoping the inequality (s.2), we finished the proof with
\begin{equation*}
    \|u_{k}\!-\!u^*\| \leq \rho^k\|u_0-u^*\|\!+\!\eta\|(H^{\top}\!-\!H_{\text{diag}})\nabla \Phi^{(2)}(y^*)\|\sum_{j=0}^{k-1}\rho^j,
\end{equation*}
where $\rho$ is defined as in Theorem \ref{theorem: no plant dynamic strongly convex}.

\subsection{Proof of Theorem \ref{theorem: no plant dynamic strongly convex vanilla}}
\label{appendix: proof 4}
Let $(\tilde u_k)_{k\in \mathbb{N}}$ and $\tilde \eta$ be the sequence of the inputs and the step size of the centralized controller \eqref{centralized update for agent}, respectively. Then,
\begin{align*}
    \norm{\tilde u_{k+1}-&\inff{u}}^2
    =\norm{\tilde u_k-\inff{u}}^2\\
    &-2\tilde \eta\nabla \obj(\tilde u_k)^{\top}(\tilde u_k-\inff{u})+{\tilde \eta}^2\norm{\nabla\obj(\tilde u_k)}^2\\
    \overset{\text{(s.1)}}{\leq}&(1-2m\tilde \eta+L^2{\tilde \eta}^2)\norm{\tilde u_k-\inff{u}}^2\\
    &+2\tilde \eta\left((H^{\top}-H_{\text{diag}})\nabla \Phi^{(2)}(\inff{y})\right)^{\top}\\
    &\cdot(\tilde \eta(\nabla\obj(\tilde u_k)-\nabla \obj(\inff{u}))-(\tilde u_k-\inff{u}))\\
    &+{\tilde \eta}^2\norm{(H^{\top}-H_{\text{diag}})\nabla \Phi^{(2)}(\inff{y})}^2\\
    \overset{\text{(s.2)}}{\leq}&(1-2m\tilde \eta+L^2{\tilde \eta}^2)\norm{\tilde u_k-\inff{u}}^2\\
    &+\tilde \eta\norm{(H^{\top}-H_{\text{diag}})\nabla \Phi^{(2)}(\inff{y})}^2\\
    &+ \tilde \eta\norm{\tilde \eta(\nabla\obj(\tilde u_k)-\nabla \obj(\inff{u}))-(\tilde u_k-\inff{u})}^2\\
    &+{\tilde \eta}^2\norm{(H^{\top}-H_{\text{diag}})\nabla \Phi^{(2)}(\inff{y})}^2\\
    \overset{\text{(s.3)}}{=}&(1-2m\tilde \eta+L^2{\tilde \eta}^2)\norm{\tilde u_k-\inff{u}}^2\\
    &+\tilde \eta(\norm{\tilde \eta(\nabla\obj(\tilde u_k)-\nabla \obj(\inff{u}))}^2\\
    &\!-\!2\tilde \eta(\nabla\obj(\tilde u_k)\!-\!\nabla \obj(\inff{u}))^{\top}(\tilde u_k\!-\!\inff{u})\!+\!\norm{\tilde u_k\!-\!\inff{u}}^2)\\
    &+({\tilde \eta}^2+\tilde \eta)\norm{(H^{\top}-H_{\text{diag}})\nabla \Phi^{(2)}(\inff{y})}^2\\
    \overset{\text{(s.4)}}{\leq}&(1-2m\tilde \eta+L^2{\tilde \eta}^2)(\tilde \eta+1)\norm{\tilde u_k-\inff{u}}^2\\
    &+({\tilde \eta}^2+\tilde \eta)\norm{(H^{\top}-H_{\text{diag}})\nabla \Phi^{(2)}(\inff{y})}^2,
\end{align*}
where in (s.1), we rewrite $\nabla \obj(\tilde u_k)=\nabla \obj(\tilde u_k)-\nabla \obj(\inff{u})+\nabla \obj(\inff{u})-\nabla^* \obj(\inff{u})$ and reorganize the terms. In (s.2), we use the inequality $2a^{\top}b \leq \norm{a}^2+\norm{b}^2, \forall a,b $. In (s.3), we expand $\norm{\tilde \eta(\nabla\obj(\tilde u_k)-\nabla \obj(\inff{u}))-(\tilde u_k-\inff{u})}^2$ and reorganize the terms. In (s.4), we use the strong convexity and smoothness of $\obj(u)$. By recursively using the above formula for $0,1,\cdots, k$, we obtain
\begin{equation}
\label{Appendix: proof theorem 4.2}
\begin{split}
    \norm{\tilde u_k -& \inff{u}}^2\leq\tilde{\rho}^k\|u_0-\inff{u}\|^2\\
    &+(\tilde \eta + {\tilde \eta}^2)\|(H^{\top}-H_{\text{diag}})\nabla \Phi^{(2)}(\inff{y})\|^2\sum_{j=0}^{k-1}\tilde{\rho}^j,
\end{split}
\end{equation}
where $\tilde{\rho}=(1-2m\tilde \eta +L^2{\tilde \eta}^2)(\tilde \eta + 1)$, and $\tilde{\rho}<1$ for a sufficiently small $\tilde \eta>0$ provided that $2m>1$. For every such $\tilde\eta$, \eqref{Appendix: proof theorem 4.2} always holds, and the left-hand side of \eqref{Appendix: proof theorem 4.2} remains the same. By analyzing the derivative of the right-hand side of \eqref{Appendix: proof theorem 4.2} with respect to $\tilde \eta$, we observe that it is monotonically increasing for $\tilde \eta > 0$. When $\tilde \eta \rightarrow 0$, we obtain the smallest possible upper bound $\norm{(H^{\top}-H_{\text{diag}})\nabla \Phi^{(2)}(\inff{y})}^2 \frac{1}{2m-1}$.

% In the steady state, i.e., $k \rightarrow \infty$, the left-hand side of \eqref{Appendix: proof theorem 4.2} (i.e., $\norm{\tilde u_k - \inff{u}}^2$) is not related to $\tilde \eta$ since $\inff{u}$ is the limit point of controller \eqref{eq: decentralized algorithm system}. However, the right-hand side of \eqref{Appendix: proof theorem 4.2} is a function of $\tilde \eta$. By analyzing the derivative with respect to $\tilde \eta$, we observe that it is monotonically increasing for $\tilde \eta > 0$. When $\tilde \eta \rightarrow 0$, it reaches its infimum $\norm{(H^{\top}-H_{\text{diag}})\nabla \Phi^{(2)}(\inff{y})}^2 \frac{1}{2m-1}$.

%%%%%%%%%%%%%%%%%%%%%%%%%%%%%%%%%%%%%%%%%%%%%%%%%%%%%%%%%%%%%%%%%%%%%%
\subsection{Proof of Theorem \ref{theorem: plant dynamic}}
\label{appendix: proof 5}
To start, we define
\begin{align*}
    \varphi_1 &= \nabla \Phi^{(1)}(u_k) + H_{\text{diag}}\nabla\Phi^{(2)}(y_k), \\
    \varphi_2 &= x_{k}-H_xu_{k}.
\end{align*}
Then, we provide a preparatory lemma that presents some useful technical results. Note that $\inff{y}\triangleq H \inff{u} + d$.

\begin{lemmaappendix}
\label{appendix lemma}
The following inequalities hold:
\begin{enumerate}
    \item $\norm{x_{k+1}-H_xu_{k+1}}^2 \leq \eta^2 \varphi_1^{\top}H_x^{\top}H_x\varphi_1 + 2\eta\sigma_{\max}(H_x^{\top}A)\norm{\varphi_1}\norm{\varphi_2} + \varphi_2^{\top}A^{\top}A\varphi_2$;
    \item $\varphi_1^{\top}(u_k-\inff{u}) \geq Nm_u\norm{u_k-\inff{u}}^2 + Nm_y\norm{y_k-\inff{y}}-NL_y\sigma_{\max}(H_{\text{diag}}-H)\norm{y_k-\inff{y}}\norm{u_k-\inff{u}}-NL_y\sigma_{\max}(C)\norm{y_k-\inff{y}}\norm{\varphi_2}$;
    \item $\norm{y_k-\inff{y}} \leq \sigma_{\max}(C)\norm{\varphi_2}+\sigma_{\max}(H)\norm{u_k-\inff{u}}$;
    \item $\norm{\varphi_1} \leq \left(NL_u+NL_y\sigma_{\max}(H_{\text{diag}})\sigma_{\max}(H)\right)\norm{u_k-\inff{u}}+NL_y\sigma_{\max}(H_{\text{diag}})\sigma_{\max}(C)\norm{\varphi_2}$.
\end{enumerate}
\end{lemmaappendix}
\begin{proof}

1): In $\norm{x_{k+1}-H_xu_{k+1}}^2$, we replace $x_{k+1}$ and $u_{k+1}$ with the expressions in \eqref{sys: LTI}. Then, the bound can be obtained by using the Cauchy-Schwarz inequality.
%in $\varphi_1$, $\varphi_2$.

2):
we can bound $\varphi_1^{\top}(u_k-\inff{u})$ by
\begin{align*}
    \varphi_1^{\top}(u_k&-\inff{u})\\
    =&(\nabla \Phi^{(1)}(u_k)+H_{\text{diag}}^{\top}\nabla\Phi^{(2)}(y_k)\\
    &-\nabla \Phi^{(1)}(\inff{u})-H_{\text{diag}}^{\top}\nabla\Phi^{(2)}(\inff{y}))^{\top}(u_k-\inff{u})\\
    \geq&Nm_u\norm{u_k-\inff{u}}^2\\
    &+ \big(\nabla\Phi^{(2)}(y_k)-\nabla\Phi^{(2)}(\inff{y})\big)^{\top}(H_{\text{diag}}(u_k\!-\!\inff{u})\\
    &+y_k-y_k+Hu_k-Hu_k+H\inff{u}-H\inff{u})\\
    \geq&Nm_u\norm{u_k-\inff{u}}^2 + Nm_y\norm{y_k-\inff{y}}^2\\
    &\!+\! \big(\nabla\Phi^{(2)}(y_k)\!-\!\nabla\Phi^{(2)}(\inff{y})\big)^{\top}((H_{\text{diag}}\!-\!H)(u_k\!-\!\inff{u})\\
    &+Hu_k-y_k)\\
    \geq&Nm_u\norm{u_k-\inff{u}}^2 + Nm_y\norm{y_k-\inff{y}}^2\\
    &- NL_y\sigma_{\max}(H_{\text{diag}}-H)\norm{y_k-\inff{y}}\norm{u_k-\inff{u}}\\
    &-NL_y\sigma_{\max}(C)\norm{y_k-\inff{y}}\norm{\varphi_2}.
\end{align*}

3):
\begin{align*}
    \norm{y_k-\inff{y}} &=\norm{y_k-Hu_k+Hu_k-\inff{y}}\\
    & \leq \sigma_{\max}(C)\norm{\varphi_2} + \sigma_{\max}(H)\norm{u_k-\inff{u}}.
\end{align*}

4):
We have
\begin{align*}
    \norm{\varphi_1&}\\
    =&\norm{\nabla \Phi^{(1)}(u_k)+H_{\text{diag}}^{\top}\nabla \Phi^{(2)}(y_k)\\
    &-\nabla \Phi^{(1)}(\inff{u})-H_{\text{diag}}^{\top}\nabla \Phi^{(2)}(\inff{y})}\\
    \leq&NL_u\norm{u_k-\inff{u}}+NL_y\sigma_{\max}(H_{\text{diag}})\norm{y_k-\inff{y}}\\
    \leq&NL_u\norm{u_k-\inff{u}}+NL_y\sigma_{\max}(H_{\text{diag}})(\sigma_{\max}(C)\norm{\varphi_2}\\
    &+ \sigma_{\max}(H)\norm{u_k-\inff{u}})\\
    \leq&(NL_u+NL_y\sigma_{\max}(H_{\text{diag}})\sigma_{\max}(H))\norm{u_k-\inff{u}}\\
    &+NL_y\sigma_{\max}(H_{\text{diag}})\sigma_{\max}(C)\norm{\varphi_2}.
\end{align*}
\end{proof}

Now we are ready to prove Theorem \ref{theorem: plant dynamic}. Consider
\begin{equation*}
\label{LTI proof step: 1}
\begin{split}
    \left\Vert\begin{matrix}x_{k+1}-H_xu_{k+1} \\u_{k+1}-\inff{u}\\\end{matrix}\right\Vert^2&=\|x_{k+1}\!-\!H_xu_{k+1}\|^2 + \|u_{k+1}\!-\!\inff{u}\|^2\\
    &=\|x_{k+1}-H_xu_{k+1}\|^2 + \|u_k-\inff{u}\|^2\\
    &\quad- 2\eta\varphi_1^{\top}(u_k-\inff{u}) + \eta^2\norm{\varphi_1}^2.
\end{split}
\end{equation*}
By using the bounds derived in Lemma \ref{appendix lemma}, we obtain
\begin{equation}
\label{appendix eq}
\left\Vert
\begin{matrix}
   x_{k+1}-H_xu_{k+1} \\
   u_{k+1}-\inff{u}\\
\end{matrix}
\right\Vert^2\leq\lambda_{\max}(\Xi)\left\Vert
\begin{matrix}
   x_k-H_xu_k \\
   u_k-\inff{u}\\
\end{matrix}
\right\Vert^2,
\end{equation}
where $\Xi = \begin{bmatrix}
        \lambda_{\max}(A^{\top}A) + a_3\eta^2 + a_4\eta& a_1\eta^2+a_2\eta\\
        a_1\eta^2+a_2\eta&1-m'\eta + L'\eta^2
        \end{bmatrix}$ and the parameters therein are defined as
\begin{align*}
    m' =&2(Nm_u+Nm_y\sigma_{\min}^2(H)\\
    &-NL_y\sigma_{\max}(H_{\text{diag}}-H)\sigma_{\max}(H)),\\
    L' =& \lambda_{\max}(H_x^{\top}H_x\!+\!I)(NL_u\!+\!NL_y\sigma_{\max}(H_{\text{diag}})\sigma_{\max}(H))^2,\\
    a_1 =& \lambda_{\max}(H_x^{\top}H_x+I)NL_y\sigma_{\max}(H_{\text{diag}})\sigma_{\max}(C)\\
    &\cdot(NL_u+NL_y\sigma_{\max}(H_{\text{diag}})\sigma_{\max}(H)),\\
    a_2 =& \lambda_{\max}(H_x^{\top}A)(NL_u+NL_y\sigma_{\max}(H_{\text{diag}})\sigma_{\max}(H))\\
    &+2Nm_y\sigma_{\max}(C)\sigma_{\max}(H)\\
    &+NL_y\sigma_{\max}(C)(\sigma_{\max}(H_{\text{diag}}-H)+\sigma_{\max}(H)),\\
    a_3 =& \lambda_{\max}(H_x^{\top}H_x+I)N^2L_y^2\sigma_{\max}^2(H_{\text{diag}})\sigma_{\max}^2(C),\\
    a_4 =& 2(\lambda_{\max}(H_x^{\top}A)\sigma_{\max}(H_{\text{diag}})\sigma_{\max}(C)\\
    &-(Nm_y\sigma_{\min}^2(C)-NL_y\sigma_{\max}^2(C))).
\end{align*}

Note that if $m'>0$ and $\eta<\frac{m'}{L'}$, by Schur's complement, the sufficient and necessary condition for $\lambda_{\max}(\Xi) <1$ is
\begin{equation*}
    (a_3m'+2a_1a_2-a_4L')\eta^2+(a_4m'+a_2^2+tL')\eta-tm'<0,
\end{equation*}
where $t = 1-\lambda_{\max}(A^{\top}A)$.
The previous inequality yields the following stability bounds
\begin{equation*}
\label{eta upper bound}
\eta < 
\begin{cases}
  \eta_1^* & \text{if $a_3m'+2a_1a_2-a_4L'>0$},\\
  \eta_2^* & \text{if $a_3m'+2a_1a_2-a_4L'\leq0$},\\
\end{cases}       
\end{equation*}
where $\eta_1^*$ and $\eta_2^*$ are defined as
\tiny
\begin{equation*}
\begin{split}
    \eta_1^*=&\frac{\sqrt{(a_4m'\!+\!a_2^2\!+\!tL')^2\!+\!4tm'(a_3m'\!+\!2a_1a_2\!-\!a_4L')}\!-\!(a_4m'\!+\!a_2^2\!+\!tL')}{2(a_3m'\!+\!2a_1a_2\!-\!a_4L')},\\
    \eta_2^* =& \frac{tm'}{a_4m'+a_2^2+tL'}.
\end{split}
\end{equation*}
\normalsize
The proof is finished by recursively applying \eqref{appendix eq}.

\end{document}